\pdfoutput=1
\RequirePackage{ifpdf}
\ifpdf % We~are running pdfTeX in pdf mode
\documentclass[pdftex]{sigma}
\else
\documentclass{sigma}
\fi

\usepackage{xspace}
\usepackage{tikz}
\usepackage[all]{xy}

\numberwithin{equation}{section}

\newtheorem{Theorem}{Theorem}[section]
\newtheorem{Corollary}[Theorem]{Corollary}
\newtheorem{Lemma}[Theorem]{Lemma}
\newtheorem{Proposition}[Theorem]{Proposition}
\newtheorem{propdefn}[Theorem]{Proposition-Definition}
 { \theoremstyle{definition}
\newtheorem{Definition}[Theorem]{Definition}

\newtheorem{Example}[Theorem]{Example}
\newtheorem{Remark}[Theorem]{Remark}
	\newtheorem{coex}[Theorem]{Counterexample}
}

	\newcommand{\nc}{\newcommand}

	\newcommand{\R}{\mathbb{R}}
	\newcommand{\K}{\mathbb{K}}

	\newcommand{\N}{\mathbb{N}}

	\nc{\C}{\mathbb{C}}

	\newcommand {\calf}{{\mathcal {F}}}

	\newcommand {\calp}{{\mathcal {P}}}

	\newcommand {\calw}{{\mathcal {W}}}

	\nc{\vep}{\varepsilon}

	\nc{\mtop}{\top}
	\nc{\grt}{\mathrm{grt}\,}		%greatest element
		
\nc{\cG}{\overline{G}}
	\nc{\mrm}[1]{{\rm #1}}
	\nc{\depth}{{\mrm d}}
	\nc{\id}{\mrm{id}}
	\nc{\Id}{\mathrm{Id}}
	\nc{\Irr}{\mathrm{Irr}}
	\nc{\Span}{\mathrm{span}}
	\nc{\mapped}{operated\xspace}
	\nc{\Mapped}{Operated\xspace}

	\nc{\ot}{\otimes}
	\nc{\Hom}{\mathrm{Hom}}
	\nc{\CS }{\mathcal{CS}}
	\nc{\bfk}{\mathbf{K}}
	\nc{\lwords}{\calw}
	\nc{\ltrees}{\calf}
	\nc{\lpltrees}{\calp}
	\nc{\Map}{\mathrm{Map}}
	\nc{\rep}{\beta}
	\nc{\free}[1]{\bar{#1}}
	\nc{\OS}{\mathbf{OS}}
	\nc{\OM}{\mathbf{OM}}
	\nc{\OA}{\mathbf{OA}}
	\nc{\based}{based\xspace}
	\nc{\tforall}{\text{ for all }}
	\nc{\hwp}{\widehat{P}^\calw}
	\nc{\sha}{{\mbox{\cyr X}}}
	\font\cyr=wncyr10 
	\nc{\Mor}{\mathrm{Mor}}

	\nc{\oF}{{\overline{F}}}
	\nc{\mge}{_{bu}\!\!\!\!{}}
	
	\nc {\conefamilyc}{{\underline{{C}}}}
	\newcommand{\loc}{locality\xspace}
	\newcommand{\Loc}{Locality\xspace}
	
	\nc{\supp}{\mathrm{supp}\ }
	\nc{\subsp}{subspace\xspace}
	
	\newcommand{\pone}{separating\xspace }
	\def \Pone {Separating\xspace }
	
	\def \ptwob {separated\xspace }

	\nc{\tsum}{\bar{+}}

\nc{\LR}{\mathrm{LVR}}
\nc{\LLR}{\mathrm{LLR}}
\nc{\LGR}{\mathrm{ LGR}}

\begin{document}
%\allowdisplaybreaks

\newcommand{\arXivNumber}{2007.03003}

\renewcommand{\thefootnote}{}

\renewcommand{\PaperNumber}{027}

\FirstPageHeading

\ShortArticleName{From Orthocomplementations to Locality}

\ArticleName{From Orthocomplementations to Locality\footnote{This paper is a~contribution to the Special Issue on Noncommutative Manifolds and their Symmetries in honour of~Giovanni Landi. The full collection is available at \href{https://www.emis.de/journals/SIGMA/Landi.html}{https://www.emis.de/journals/SIGMA/Landi.html}}}

\Author{Pierre CLAVIER~$^{\rm a}$, Li GUO~$^{\rm b}$, Sylvie PAYCHA~$^{\rm a}$ and Bin ZHANG~$^{\rm c}$}

\AuthorNameForHeading{P.~Clavier, L.~Guo, S.~Paycha and B.~Zhang}

\Address{$^{\rm a)}$~Institute of Mathematics, 	University of Potsdam, D-14476 Potsdam, Germany\\
\hphantom{$^{\rm a)}$}~(S.~Paycha on leave from Universit\'e Clermont-Auvergne, Clermont-Ferrand, France)}
\EmailD{\href{mailto:clavier@math.uni-potsdam.de}{clavier@math.uni-potsdam.de}, \href{mailto:paycha@math.uni-potsdam.de}{paycha@math.uni-potsdam.de}}

\Address{$^{\rm b)}$~Department of Mathematics and Computer Science, Rutgers University,\\
\hphantom{$^{\rm b)}$}~Newark, NJ 07102, USA}
\EmailD{\href{mailto:liguo@rutgers.edu}{liguo@rutgers.edu}}

\Address{$^{\rm c)}$~School of Mathematics, Yangtze Center of Mathematics,	Sichuan University,\\
\hphantom{$^{\rm c)}$}~Chengdu, 610064, China}
\EmailD{\href{mailto:zhangbin@scu.edu.cn}{zhangbin@scu.edu.cn}}

\ArticleDates{Received July 06, 2020, in final form March 02, 2021; Published online March 22, 2021}

\Abstract{After some background on lattices, the locality framework introduced in earlier work by the authors is extended to cover posets and lattices. We~then extend the correspondence between Euclidean structures on vector spaces and orthogonal complementations to a~one-one correspondence between a class of locality structures and orthocomplementations on bounded lattices. This recasts in the context of renormalisation classical results in lattice theory.}


\Classification{06C15; 08A55; 81T15; 15A63}

\begin{flushright}
\begin{minipage}{70mm}
\it This paper is dedicated to Gianni Landi\\ on the occasion of his sixtieth birthday
\end{minipage}
\end{flushright}

\renewcommand{\thefootnote}{\arabic{footnote}}
\setcounter{footnote}{0}

\section{Introduction}
\subsection{The general setup and our aims}

The notion of complementation, or complement map~-- roughly speaking, an operation which separates a subset $M$ of a given set $X$ from another
subset $M'$, its complement, so that the information on $X$ is split into the part on $M$ and the part on the complement $M'$~-- provides a \pone tool that is ubiquitous in mathematics.
For example, the notion of (ortho)complementation naturally arises in the context of {\it axiomatic quantum mechanics}, where various types of binary relations are used, that are defined on the set of all questions testable
for a given physical system~\cite{CN}.
Also, as we shall see below, the notion of complementation
 plays a central role in {\it renormalisation procedures}. Typical examples of complementations are the set complementations and the orthogonal complementations, taking respectively a subset of~a~reference set to its complement in the set, a linear subspace of~a~reference Euclidean vector space to its orthogonal complement in the space.

Our guiding example throughout this paper is the set $(G(V), \preceq)$ of linear subspaces in a finite dimensional vector space
$V$ equipped with the partial order $\preceq $ corresponding to the inclusion of~linear subspaces.

 An inner product $Q$ on $V$ gives rise to a
complementation on $G(V)$:
\begin{gather*}
\Psi^Q\colon\quad	G(V)\longrightarrow G(V),\qquad
	W \longmapsto	 W^{\perp^Q},
\end{gather*}
 where $W^{\perp^Q}:=\{v\in V\,|\, Q(v, w)=0,\, \forall\, w\in W\}$ is the $Q$-orthogonal complement of~$W$. It also gives rise to a symmetric binary relation on $G(V)$, namely the orthogonality relation
\begin{gather}
\label{eq:introQW}
W_1\perp^Q W_2\Longleftrightarrow Q(w_1, w_2)=0\qquad
\forall\, (w_1, w_2)\in W_1\times W_2,
\end{gather}
and we have
\begin{gather*}
W_1\perp^Q W_2\Longleftrightarrow W_2\subseteq \Psi^Q(W_1),\qquad
W_2=\Psi^Q(W_1)\Longleftrightarrow W_2=\grt \big\{ W \preceq V\,|\, W\perp^QW_1\big\}.
\end{gather*}
Here $\grt $ means the greatest element by inclusion. This establishes a one-one correspondence between locality relations and orthocomplementations given by scalar products:
\begin{gather}\label{eq:correspQ}
\perp^Q\longleftrightarrow \Psi^Q.
\end{gather}
The symmetric binary relation $\perp^Q$ on $G(V)$ is a particular instance of the more general notion of~{\it locality relation} introduced in~\cite{CGPZ1}. Since the poset $(G(V), \preceq)$ equipped with the intersection and the sum (of two vector spaces) is a lattice, the one-one correspondence~\eqref{eq:correspQ} serves as a motivation to investigate the relation between complementations and \loc relations on~lattices.

So we ask, under what conditions one can derive on a lattice, a \loc relation from a~com\-p\-le\-men\-tation.

Theorem \ref {thm:complement_strongly_local} provides an answer to this question:
for any bounded lattice, there is a one-one correspondence
\begin{gather}\label{eq:corresp}
\text{orthocomplementations} \quad \longleftrightarrow\quad \text{strongly \pone \loc relations.} \end{gather}
When applied to
	the lattice $G(V)$, this generalises the correspondence (\ref{eq:correspQ}) between orthogonality and orthogonal complement on vector spaces.
	
We were informed by a referee report on a previous version of this paper, that this one-one correspondence was already known and proven in~\cite{CM}, this leading us to some substantial reformulations and restructuring.\footnote{In particular, we learned that a poset with locality amounts to a weak degenerate orthogonal poset in the sense of \cite[Definition 2.1]{CM}. We~nevertheless use a slightly different terminology which we find well-suited for the locality setup we have in mind.}\ Although the one-one correspondence~\eqref{eq:corresp} seems to be common knowledge in the lattice community, we believe that recasting this known result in the context of renormalisation is relevant for the mathematical physics community. We~feel that this explo\-ratory and survey type article serves to promote the notion of orthocomplements beyond the lattice theory community, in the mathematical physics community where complements are constantly used in the context of renormalisation, as we shall now explain.

\subsection{ Locality and complements in renormalisation}

Complementations play an essential role in renormalisation, where they arise in various disguise and are used to separate divergent terms from convergent terms.

{\samepage
In A.~Connes and D.~Kreimer's algebraic Birkhoff factorisation approach to renormalisation
\cite{CK,M}, the coproduct is typically built from a (relative) complementation on a poset $(X, \leq)$ by means of
$\Delta x=\sum_{y} x \otimes (x\setminus y)$, where $ x\setminus y$ is a complement to an element $y\leq x$. This holds for the coproduct on rooted trees for which we can view the crown of
a rooted tree as the complement of its trunk (a~sub-rooted tree) after an admissible cut (see, e.g.,~\cite{F1,F2}), for the coproduct on~graphs, with the contracted graph $\Gamma\backslash \gamma$ corresponding to
the complement of a~subgraph $\gamma$ inside a connected 1 particle irreducible Feynman graph $\Gamma$ (see, e.g.,~\cite{M}).

}

%\looseness=1
Complementations also arise in generalised Euler--Maclaurin formulae, which relate
sums to~integrals. A systematic choice of complement (called a rigid complement) of a linear subspace of~a vector space was used in~\cite{GP} to interpolate between exponential sums and
exponential integrals over rational polytopes in a rational vector space. A notion of ``transversal cone $ C\setminus F$ to a face $F$ of a cone $C$" was used as a complementation $F\mapsto C\setminus F$ by N.~Berline and M.~Vergne~\cite{BV}, to prove a local Euler--Maclaurin formula on polytopes. Cones form a poset for the relation ``$F\leq C$ if $F$ is a face of $C$", and we could reinterpret the Euler--Maclaurin formula on cones, as an algebraic Birkhoff-factorisation by means of the coproduct $\Delta C=\sum_{F\leq C} (C\setminus F)\otimes F$~\cite{GPZ1}.

Renormalisation issues also underly our quest for a description of complementations on (finite dimensional) vector spaces. To explain how this motivates our comparative study of locality relations and complementations, let us first describe an abstract framework for a renormalisation scheme in the context of locality structures~\cite{CGPZ1, CGPZ2}:
\begin{itemize}\itemsep=0pt
\item an (locality) algebra $ \left({\mathcal A}, \top_A, m_A\right)$ ($m_A$ stands for the product and $\top_A$ for the locality relation) might it be of Feynman graphs, trees, or cones,
\item an algebra of meromorphic germs at zero $\left({ {\mathcal M} } , \cdot\right)$ might it be the algebra ${\mathcal M}(\C)$ of meromorphic germs as in A.~Connes and D.~Kreimer's algebraic Birkhoff factorisation approach~\cite{CK} or the algebra ${\mathcal M}(\C^\infty)$ of multi-variable meromorphic germs with linear poles studied in~\cite{GPZ2},
\item a (locality) morphism $ \Phi\colon ({\mathcal A}, \top_A, m_A)
\longrightarrow({\mathcal M}, \cdot)$ such as Feynman integrals~\cite{DZ}, branched zeta functions~\cite{CGPZ3, CGPZ4} and conical zeta functions~\cite{GPZ1}.
		\end{itemize}
The locality principle translates to the partial multiplicativity of $\Phi$:
\begin{gather*}
a_1 \top_A a_2\Longrightarrow \Phi(m_A(a_1 ,a_2))= \Phi (a_1) \cdot \Phi(a_2).
\end{gather*}
Renormalising consists in building a (locality) character $\Phi^{\rm ren}\colon ({\mathcal A}, \top_A, m_A)\longrightarrow (\C, \cdot)$
\begin{gather*}
a_1 \top_A\, a_2\Longrightarrow \Phi^{ \rm ren}({ m_A}( a_1 ,a_2))= { \Phi}^{ \rm ren} ( a_1 ) \cdot{ \Phi}^{\rm ren}(a_2).
\end{gather*}

 To build $ \Phi^{ \rm ren} $, one first needs to separate the holomorphic part ${ \Phi}_+$ from the polar part ${ \Phi}_-$ of ${ \Phi}$ and then to evaluate it at zero setting $\Phi^{\rm ren}:= \Phi_+(0)$. This splitting relies on the splitting $ {\mathcal M} = {\mathcal M}_+\oplus {\mathcal M}_-$ of the algebra of meromorphic germs at zero into a holomorphic part $ {\mathcal M}_+$ and a polar part ${\mathcal M}_-$. The map $\Phi_+$ is built:
\begin{itemize}\itemsep=0pt
\item in one variable by means of an algebraic Birkhoff factorisation following A.~Connes and
D.~Kreimer~\cite{CK},
\item in several variables as $\Phi_+:=\pi_+\circ \Phi$ via a projection map $\pi_+$ onto
${\mathcal M}_+$ (see, e.g.,~\cite{CGPZ2}).
\end{itemize}
			We adopt the second approach which can be interpreted as a minimal subtraction scheme in~several variables. To build the projection map $\pi_+$, we use Laurent expansions of meromorphic germs in several variables with linear poles (see~\cite{GPZ2}) whose construction in turn requires a~(filtered) separating device on the underlying spaces $V=\C^k$, $k\in \N$.
			
The resulting renormalised map $\Phi^{\rm ren}=\pi_+\circ \Phi$ depends on the choice of the projection $\pi_+$, which in turn is dictated by the choice of splitting $ {\mathcal M} = {\mathcal M}_+\oplus {\mathcal M}_-$. The passage from one splitting to another splitting is encoded in a renormalisation group which we hope to describe in forthcoming work. In~\cite{GPZ2} we built Laurent expansions using the orthogonality relation $Q$ on $V:=\C^k$ as a separating device, leading to an orthogonal projection $\pi_+^Q$. The present paper investigates alternative separating devices on $V $ via separating devices on the lattice $G(V)$ which we intend to use to extend our construction of Laurent expansions beyond the ones obtained by orthogonal splittings. The group of transformations of meromorphic germs that preserve holomorphic ones plays a central role in the context of renormalisation since its elements are transformations which take one renormalised value to another and can therefore be interpreted as elements of a hypothetical renormalised group.

\subsection{Plan of the paper}

Section~\ref{section:lattices} is a review of the basics of lattice theory, such as distributivity (Proposition-Defi\-ni\-tion~\ref{propdef:distrib1}) and modularity (Definition~\ref{defn:modularlattice}) of lattices, which we spell out in order to later generalise them to the locality setup.

Section~\ref{sec:loclat} is dedicated to lattices equipped with a locality relation, which we call locality lattices (Definition~\ref{def:loclattice}). We~first define and characterise locality posets (Proposition-Definition~\ref{de:locposet}) after which we characterise locality lattices (Proposition~\ref{prop:locvee}).
Whereas the subspace lattice $G(V)$ is not a
\loc lattice for the ``disjointedness" \loc relation $W_1\top W_2\Leftrightarrow W_1\cap W_2=\{0\}$, it is for the \loc relation
$\perp^Q$ of (\ref{eq:introQW}), see Example~\ref{ex:GVlocality}.

We then review the notion of complement (Definition~\ref{defn:relcomplt}) and orthocomplement (Definition~\ref{defn:orthocomplemented_lattice}) on lattices and discuss the strongly separating property (Proposition~\ref{prop:Psistronglysep}) of lattices with orthocomplement. This is later used to classify orthocomplements on the subspace lattice of $\R^2$ (Corollary~\ref{coro:PsiR2}).

Alongside the separating property of orthocomplementations singled out in Section~\ref{sec:ortholatt}, in~Sec\-tion~\ref{sec:loclat} we single out separating locality relations (Definition~\ref{def:splitlattice}) on lattices, after which we introduce the more stringent strongly separating locality relations (Definition~\ref{def:strongsplitlattice}). We~then prove the equivalence between orthocomplementations and strongly separating locality lattices (Theorem~\ref{thm:complement_strongly_local}). An easy consequence is the classification of strongly separating locality relations on the lattice $G\big(\R^2\big)$ in Corollary~\ref{cor:topGR2}.

Section~\ref{sect:locvectspaces} is dedicated to our guiding example, the modular bounded lattice $(G(V), \preceq)$. We~show (Proposition~\ref{pp:lvtoll}) that locality relations on the lattice $G(V)$ are in one-one correspondence with locality relations on the vector space $V$ introduced in~\cite{CGPZ1}.
	Specialising to non-degenerate locality relations on a vector space $V$ (Definition~\ref{de:vpproper}), we show that these are in one-one correspondence with strongly separating locality relations on~$G(V)$ (Proposition~\ref{prop:VGV2}). Typical non-degenerate relations are orthogonality relations, and the notion of orthogonal basis generalises to that of locality basis (Definition~\ref{defn:compl_map_TVS}). A Gram--Schmidt type argument is used to show that a strongly \pone \loc relation on $G(V)$ admits a \loc basis (Proposition~\ref{prop:Psisplitting}). The fact that a given basis can be the \loc basis for multiple \loc relations suggests the richness of locality relations on a vector space.

	Following a referee's suggestion, in an appendix, we show a correspondence between another class of locality relations and of orthocomplementations. For complete atomistic lattices, we present a one-one correspondence between a class of locality lattice relations and certain locality relations on the set of atoms of the lattice, which applies to the subspace lattice considered in this paper.

We have therefore reached our goal in extending the correspondence (\ref{eq:correspQ}) well beyond loca\-lity relations of the type $\perp^Q$
of (\ref{eq:introQW}), showing the more general correspondence (\ref{eq:corresp}).
This way, we could detect strongly
\pone \loc relations on vector spaces beyond the ortho\-go\-na\-lity \loc which classifies strongly \pone \loc relations on $\R^2$, among which lies the orthogonality \loc relations corresponding to
(\ref{eq:introQW}).

\section {Modular lattices} \label{section:lattices}
{\samepage
This section puts together notions and examples on lattices to provide background
for our later study of lattices in a locality setup. For references on this background material, see, e.g.,~\cite{B,FF, Gra2,Gra1,HP,K,ML,MM,PR}.

}

\subsection{Lattices}
%\label{s:distr}
For completeness, let us first recall that a {\it lattice} is a partially ordered set (poset) $(L,\leq)$, any two-element subset $ \{a, b\}$ of which has a
least upper bound (also called a join) $a\vee b$, and a~greatest lower bound (also called a meet) $a\wedge b$ such that
\begin{enumerate}\itemsep=0pt
\item[$(a)$] both operations are associative and monotone with respect to the order,
\item[$(b)$] if $a_1\leq b_1$ and $a_2\leq b_2$, then $a_1\wedge a_2\leq b_1\wedge b_2$ and $a_1\vee a_2\leq b_1\vee b_2$.
\end{enumerate}
We shall sometimes write $(L,\leq,\wedge, \vee)$ for completeness.

A {\it morphism} $\varphi\colon L\to L'$ of two lattices $(L,\leq,\wedge, \vee)$ and $(L',\leq',\wedge', \vee')$ is a morphism $\varphi\colon (L,\leq )$ $\to (L', \vee')$ of posets compatible with the operations
\begin{gather*}
\varphi(a)\wedge'\varphi(b)= \varphi(a\wedge b)
\end{gather*}
and
\begin{gather*}
\varphi(a)\vee'\varphi(b)= \varphi(a\vee b) \qquad \forall\, (a, b)\in L^2.
\end{gather*}

\begin{Example} \label{ex:powerset} %\label{ex:gcdlcm}
A first example of lattices is the power set ${\mathcal P}(X)$ of a set $X$ with inclusion as the partial order. Then $\vee$ is the union and $\wedge$ is the intersection. Another elementary example is $\N$ with the partial order $a\vert b\Longleftrightarrow \exists\, k\in \N$, $b=a k$. Then $\vee$ is the least common multiple and~$\wedge$~is the greatest common divisor.
\end{Example}

Here are central examples for our purposes.
\begin{Example}\label{ex:GV}\qquad
\begin{enumerate}\itemsep=0pt
\item[$(a)$] %\label{it:sp}
 Given a finite dimensional vector space $V$, let $G(V)$ denote the set of linear subspaces of~$V$ equipped with
 the partial order ``to be a linear subspace of'' denoted by $\preceq$. The lattice $\left(G(V), \preceq\right)$, that we call the {\it \subsp lattice} comes equipped with
 the sum $\vee=+$ and the intersection $\wedge=\cap $ as lattice operations and we write $(G(V), \preceq, \cap, +)$.
\item[$(b)$]
Given a Hilbert space $V$, let $\cG(V)$ denote the set of {\it closed} linear subspaces of $V$ equipped with
the partial order ``to be a closed linear subspace of" denoted by $\preceq$. Since the sum of two closed linear subspaces of $V$ is not necessarily closed, the topological sum $W\tsum W'$ of two closed linear subspaces $W$ and $W'$ of $V$ is the topological closure $\overline{W+W'}$ of their algebraic sum. So an element $v$ lies in the topological sum of $W $ and $W'$ in $G(V)$ if it can be written as the limit $v=\underset{n\to \infty}{\lim}(w_n+w_n')$ of a sum of elements $w_n\in W$, $w_n'\in W'$. If $W$ and $W'$ are finite dimensional, their topological sum coincides with their algebraic sum.
The lattice $\left(\cG(V), \preceq\right)$ with the topological sum $\vee:=\tsum$ and the intersection $\wedge:=\cap $ is a~lattice and we write $(\cG(V), \preceq, \cap, \tsum)$ and call the {\it closed \subsp lattice}. \looseness=1
%\label{it:tsp}
\end{enumerate}
The study of such lattices was motivated by quantum mechanics, see, e.g.,~\cite{CL,EGL} and the references therein. The mathematical interpretation of quantum theory is indeed based on the structure of the set $\overline G(V)$ of a given Hilbert space $V$, or equivalently of projection operators, viewed as events and on the probabilistic interpretation of Hermitean operators, viewed as observables,
see, e.g.,~\cite[Section~16]{K}.
\end{Example}

\begin{Remark}
With the applications to renormalization in mind, we will later focus on \subsp lattices of finite
dimensional vector spaces though many results have counterparts for
sets of~closed linear subspaces of topological vector spaces. We~hope to investigate topological aspects systematically in forthcoming work.
\end{Remark}

\begin{Definition} \label{de:latticenotion}\qquad\samepage
\begin{enumerate}\itemsep=0pt
\item[$(a)$]
A {\it poset ideal} of a poset $(P,\leq)$ is a subset $M$ of $P$ such that if $a\in M$ and $b\leq a$, then $b\in M$. 	
\item[$(b)$] The smallest ideal of a poset $(P,\leq)$ that contains a given element $a$ is called the {\it principal $($poset$)$ ideal} of $(P,\leq)$ generated by $a$. It is given by
	\begin{gather*}%\label{eq:La}
\downarrow a:=\{b\in P\,|\, b\leq a\}.
	\end{gather*}
\item[$(c)$]
	A {\it sublattice} of a lattice $(L,\leq, \wedge,\vee)$ is
a subset $M$ of $ L$ closed under the meet and join operations in~$M$, i.e., such that for every pair of elements $a, b$ in $M$ both $a \wedge b$ and $a\vee b$ are in $M$.
\item[$(d)$]
An {\it ideal} of a lattice $(L, \leq,\wedge,\vee)$ or {\it lattice ideal} is a poset ideal and a sublattice, that is, a~poset ideal such that $a\vee b$ lies in $I$ for any $(a, b)\in I^2$.
\item[$(e)$]
The smallest ideal of a lattice $(L, \leq,\wedge,\vee)$ that contains a given element $a$ is called the {\it principal $($lattice$)$ ideal} generated by $a$. It coincides with the principal poset ideal $\downarrow a=\{b\in L\,|\, b\leq a\}=\{a\wedge b\,|\, b\in L \}$.
\item[$(f)$]
Any {\it interval}
\begin{gather*}
[a, b]:=\{c\in L\,|\, a\leq c\leq b \}
\end{gather*} in a lattice $L$ is itself a lattice.
\item[$(g)$]
	A lattice $(L,\leq,\wedge, \vee)$ is {\it bounded} from above (resp.~from below)) if it has a {\it greatest element} (also called top element)
	denoted by $1$ (resp.~a {\it least element} (also called bottom) denoted by $0$), which satisfies $ x\leq 1$ (resp.~$0\leq x$) for any $x\in L$.
	If it is bounded both from below and from above in which case $L=[0, 1]$, we call it {\it bounded} and use the short hand notation $(L, \leq, 0, 1)$. We~will typically consider bounded lattices.
\end{enumerate}
\end{Definition}

\subsection{Distributivity and modularity}
 Let $L$ be a lattice. For $a, b, c\in L$, we have
 \begin{gather*}%\label{eq:partdist}
 (a\wedge b)\vee (a\wedge c)\leq a\wedge (b\vee c) \qquad {\rm and}\qquad a\vee (b\wedge c)\leq (a\vee b)\wedge(a\vee c),
 \end{gather*}
 yet the operations $\vee$ and $\wedge$ are not necessarily distributive with respect to each other.

\begin{propdefn}[{\cite[Section 4, Lemma 10]{Gra2}}] \label{propdef:distrib1}
A lattice $L$ is called distributive if it fulfills one of the two equivalent properties
 \begin{gather}\label{eq:distr1}
 a\wedge (b\vee c)= (a\wedge b)\vee (a\wedge c), \qquad \forall\, a,b, c\in L,
 \end{gather}
 or the dual identity:
 \begin{gather*}%\label{eq:distr2}
 a\vee (b\wedge c)= (a\vee b)\wedge(a\vee c), \qquad \forall\, a,b, c\in L.
 \end{gather*}
 \end{propdefn}

 \begin{Example} The power set lattice and the lattice of positive integers in Example~\ref{ex:powerset} are distributive.
 \end{Example}

\begin{coex} \label{coex:GR2} The \subsp lattice $G(V)$ introduced in Example~\ref{ex:GV} is not distributive. In~$G\big(\R^2\big)$, setting $W_1= \langle e_1\rangle$, $W_2= \langle e_2\rangle$ and $W=\langle e_1+e_2\rangle$, we have
	$W=(W_1+ W_2)\cap W\neq (W_1\cap W)+ (W_2\cap W)=\{0\}$ and $W=(W_1 \cap W_2)+ W\neq (W_1+ W)\cap (W_2+ W)=\R^2$.
\end{coex}

The following examples provide simple nondistributive lattices.

\newpage

\begin{Example}[{\cite{Gra1}}]\label{ex:diamond}
	The {\it diamond lattice} $M_3=\{0,a,b ,c,1\}$ with $a$, $b$, $c$ pairwise incomparable is not distributive since $(a\wedge b)\vee c= c$ whereas
$(a\vee c)\wedge (b\vee c)= 1$.
\begin{figure}[h!]\centering
\begin{tikzpicture}[scale=.5]
\node (one) at (0,2) {$1$};
\node (a) at (-2,0) {$a$};
\node (b) at (0,0) {$b$};
\node (c) at (2,0) {$c$};
\node (zero) at (0,-2) {$0$};
\draw (zero) -- (a) -- (one) -- (b) -- (zero) -- (c) -- (one);
\end{tikzpicture}
\caption{The Hasse diagram of the diamond lattice.}
\end{figure}
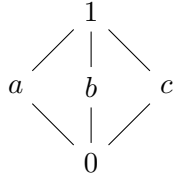
\end{Example}

\begin{Example}\label{ex:pentagon}
 The {\it pentagon lattice} $N_5:=\{0,b_1,b_2,c,1\}$ with partial order defined by
$0\leq b_1< b_2\leq1$ and $0\leq c\leq 1$ with $b_i$ and $c$ incomparable, is not distributive since $b_2\wedge (b_1\vee c)=b_2\wedge 1=b_2$ while $(b_2\wedge b_1)\vee(b_2\wedge c)=b_1\vee 0=b_1$.
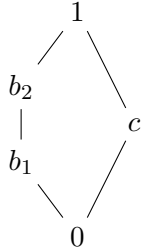
\begin{figure}[h!]\centering
\begin{tikzpicture}[scale=.5]
\node (one) at (0,2) {$1$};
\node (a) at (-1.5,0) {$b_2$};
\node (b) at (-1.5,-2) {$b_1$};
\node (c) at (1.5,-1) {$c$};
\node (zero) at (0,-4) {$0$};
\draw (zero) -- (c) -- (one) -- (a) -- (b) -- (zero);
\end{tikzpicture}
\caption{The Hasse diagram of the pentagon lattice.}
\end{figure}
\end{Example}

These two examples turn out to be the cores of any non-distributive lattices.

\begin{Proposition}[{\cite[Chapter~IX, Theorem~2]{B}, \cite[Theorems~101 and 102]{Gra1}}]\label{prop:charnondist}
A lattice is distributive if and only if it does not contain a pentagon or a diamond as a sublattice.
\end{Proposition}

Here is a useful characterisation of distributivity.
\begin{Proposition}[{\cite[Theorem~I.10]{B}, \cite[Corollary~3.1.3]{PR}}]%\label{prop:distrcanc}
A lattice $L$ is distributive if and only if the following {\it cancellation law} holds:
\begin{gather*}%\label{eq:canclaw}
({\it cancellation\ law})\qquad a\wedge c= b\wedge c\qquad {\rm and}\qquad a\vee c= b\vee c \Longleftrightarrow a=b.
\end{gather*}
 \end{Proposition}

We recall the following definition, which can be viewed as a relative distributivity.

\begin{Definition}[{\cite[Section 4, Lemma 12]{Gra2}}] \label{defn:modularlattice}
A lattice $L$ is called {\it modular} if
\begin{gather*}
({\it modularity})\qquad 	a\geq c \Rightarrow (a\wedge b)\vee c= a \wedge (b\vee c).
%\label{eq:modlat}
\end{gather*}
\end{Definition}

\begin{Example}
A distributive lattice is modular since modularity is a special case of the distributivity in~\eqref{eq:distr1}. In~particular,
$({\mathcal P}(X), \subseteq, \cap, \cup)$ and $(\N, \vert, \wedge, \vee)$ are modular.
	\end{Example}

Not every modular lattice is distributive.

\begin{Example}
The diamond lattice of Example~\ref{ex:diamond} is modular and not distributive.
\end{Example}
	
	The following well-known example will be crucial for our applications:
	\begin{Example}\label{ex:GVmodularity}
		For any finite dimensional vector space $V$, the \subsp lattice $G(V)$ introduced in Example~\ref{ex:GV} is modular, yet it is not distributive when $\dim V>1$ (see
		Counterexample~\ref{coex:GR2}).
	\end{Example}
\begin{coex}
The pentagon lattice of Example~\ref{ex:pentagon} is not modular.
	\end{coex}

Modularity is a hereditary property, which leads to a refinement of Proposition~\ref{prop:charnondist}.

\begin{Proposition}[{\cite[Theorems~101 and 102]{Gra1}}]%\label{pp:sublatticemod}
\qquad
\begin{enumerate}\itemsep=0pt
\item[$(a)$]
A lattice is modular if and only if it does not contain a pentagon sublattice.
\item[$(b)$]
A modular lattice is distributive if and only if it does not contain a diamond sublattice.
\end{enumerate}
\end{Proposition}

We have the following relationship between modularity and the cancellation law.
\begin{Proposition}[{\cite[Corollary~2.1.1]{PR}}]
	%\label{prop:modcanc}
A lattice $L$ is modular if and only if it obeys the following {\it modular cancellation law}:
\begin{gather*}
\text{for any } (a, b, c)\in L^3, \quad\ \text{if } a\leq b,\, a\wedge c=b\wedge c\quad \text{and}\quad a\vee c=b\vee c,\quad\ \text{then } a=b.
	%\label{eq:modcanc}
	\end{gather*}
\end{Proposition}

To sum up we have the following correspondences:
\begin{gather*}
\begin{split}
\xymatrix{
	\text{distributivity} \ar@2{<->}[rr] \ar@2{->}[d] &&
	\text{cancellation law} \ar@2{->}[d]\\
		\text{modularity} \ar@2{<->}[rr] &&
	\text{modular cancellation law}
}
\end{split}
%\label{diag:distcanc}
\end{gather*}

\section{\Loc relations on lattices}\label{sec:loclat}

To study properties in lattices equipped with \loc relations, we first equip posets with \loc relations.

As in~\cite{CGPZ2}, we call {\it \loc relation} on a set $X$, any symmetric binary relation $\top$ on $X$ and the pair $(X, \top)$ a {\it locality set}.
For $A\subseteq X$, we call
\begin{gather}\label{eq:polar}
A^\top:=\{x'\in X\,|\, x\top x',\ \forall\, x\in A\}
\end{gather}
the {\it polar set of $A$}.

We observe that for any element $a$ in a locality set $(X, \top)$ we have $a\in \big(a^\top\big)^\top$.

\begin{propdefn} \label{de:locposet}
	A {\it \loc poset} is a poset $(P, \leq)$ equipped with a \loc relation on the set $P$ that satisfies one of the following equivalent conditions which amount to a~{\it compa\-ti\-bi\-lity condition} with the partial order
	\begin{enumerate}\itemsep=0pt
\item[$(a)$] %\label{de:poset_i}
if $a\leq b$, then $b^\top\subseteq a^\top, $ 	

\item[$(b)$] %\label{de:poset_ii}
for any $c\in P$, the set $c^\top$ is a {\it poset ideal} of $(P,\leq)$.
\item[$(c)$] %\label{de:poset_iii}
$\downarrow a\subseteq \big(a^\top\big)^\top $, i.e., if $b\leq a$ then $b\in \big(a^\top\big)^\top $.
	\end{enumerate}

Then the relation $\top$ is called a {\it poset \loc relation}.
\end{propdefn}

Checking the equivalences $(a)\Leftrightarrow (b)\Leftrightarrow (c)$ is an easy exercise.

\begin{Remark}\qquad
\begin{enumerate}\itemsep=0pt
\item[$(a)$] A locality poset %\zb {we named it "locality poset" already}
amounts to a weak degenerate orthogonal poset in the sense of \cite[Defi\-nition 2.1]{CM} (see also \cite[supplementary Remark~6 in~Section~16]{K}, which refers to~\cite{Sz1}). There, the set ${\rm Ker} \top:=a^\top\cap \{a\}=\{a\in L, \, a\top a\}$ is called the kernel of $\top$.
\item[$(b)$] If $P$ has a bottom element $0$ (resp.~top element $1$), then $p^\top\subseteq 0^\top $ (resp.~$p^\top\supseteq 1^\top $) for any $p\in P$. We~will later consider lattices for which $0^\top=P$ (resp.~$P^\top=\{0\}$).
\end{enumerate}
\end{Remark}

{\samepage\begin{Example} \label{ex:locposetex}
Some locality posets % \zb {we named it "locality poset" already}
are
\begin{enumerate}\itemsep=0pt
\item[$(a)$] the power set $\left({\mathcal P}(X), \subseteq \right)$ of Example~\ref{ex:powerset} endowed with the \loc relation
	\begin{gather*}
	A\top_\cap B\qquad\text{if and only if}\qquad A\cap B=\varnothing,
	\end{gather*}
\item[$(b)$] the \subsp poset $(G(V), \preceq)$ of Example~\ref{ex:GV}
	endowed with the \loc relation
	\begin{gather*}
	W_1\top_\cap W_2\qquad\text{if and only if}\qquad W_1\cap W_2=\{0\}.
	\end{gather*}
\end{enumerate}
	\end{Example}

}

Just as in Proposition-Definition~\ref{de:locposet}, we required from a locality relation $\top$ on a poset, that the polar sets be poset ideals, from a locality relation on a lattice we require that polar sets be lattice ideals.

\begin{Definition} \label{def:loclattice} A {\it \loc relation on a lattice} $(L\leq,\wedge,\vee)$ is a \loc relation $\top$ on the set $L$ such that, for any $a\in L$, the polar set $a^\top$ (defined by~\eqref{eq:polar}) is a lattice ideal of $L$.

	We call the
quintuple $(L, \leq, \top, \wedge, \vee) $ a {\it \loc lattice}.
\end{Definition}
\begin{Remark} %\label{Remark:intersection_localities}
 It is easy to see that the intersection of poset ideals is a poset ideal. Similarly, the intersection of lattice ideals are lattice ideals. Therefore the intersection of locality relations on a lattice is still a locality relation on the same lattice.
 \end{Remark}

Since a lattice ideal is a poset ideal, a \loc lattice is a \loc poset.

Note that the operations $\vee$ and $\wedge$ are defined on the whole cartesian product $L\times L$. \begin{Remark}A related notion is a partial lattice defined in~\cite[Section~I5.4]{Gra1}.
	\end{Remark}

\begin{Example} \label{ex:GVlocality} Given a Hilbert space $(V, \langle \cdot, \cdot, \rangle)$, the closed \subsp lattice $\big(\cG(V),\preceq,\tsum,\cap\big)$ introduced in Example~\ref{ex:GV} is a \loc lattice for the \loc relation: $U_1\top U_2$ if $\langle u_1, u_2\rangle=0$, $\forall\, u_i\in U_i$, $i=1,2$.

To show this, let $A_1$, $A_2$, $B$ be closed linear subspaces of $(V, \langle\cdot, \cdot\rangle)$ with $A_i\top B$ for $i\in\{1,2\}$. The fact that $(A_1\cap A_2) \top B$ is
straightforward. To show the relation $(A_1\tsum A_2) \top B$, for any $a\in A_1\tsum A_2$ and $b\in B$ we write $a=\underset{n\to \infty}{\lim}(a_1(n)+a_2(n))$ and $b=\underset{n\to \infty}{\lim} b(n)$ with $a_i(n)\in A_i$ and $b(n)\in B$ for any $n\in \N$.
By the bilinearity and continuity of the inner product, we have
\begin{gather*}
\langle a,b\rangle = \underset{m,n\to \infty}{\lim}\langle a_1(m)+a_2(m),b(n)\rangle = \sum_{i=1}^2\underset{m,n\to \infty}{\lim}\langle a_i(m),b(n)\rangle=0.
\end{gather*}
Therefore $(A_1\tsum A_2) \top B$ and $(A_1\tsum A_2) \top (B_1\tsum B_2)$ for $A_1$, $A_2$, $B_1$, $B_2$ with $A_i\top B_j$ for $i,j\in\{1,2\}$ as required by symmetry of $\top$. Finally, $\cG(V)$ has a least element, the trivial vector space $\{0\}$, and a greatest
element, the full vector space $V$. Trivially, we have $\{0\}\top G(V)$ and, for any $V\in \cG(V)$, if $A\top V$ then $ A=\{0\}$.
\end{Example}
In particular,
\begin{Example} %\label{ex:GVfdlocality}
The subspace lattice $(G(V), \preceq, +, \cap)$ on a finite dimensional Euclidean real (resp.~Hermitian complex) vector space $(V, \langle \cdot, \cdot\rangle)$ comes equipped with a lattice locality $U_1\top U_2$ if $\langle u_1, u_2\rangle=0$ $\forall\, u_i\in U_i$, $i=1,2$.
	\end{Example}

\begin{Proposition}\label{prop:locvee} A lattice $(L, \wedge, \vee, \top)$ equipped with a poset \loc relation is a \loc lattice if and only if for any finite index set $I$ and $a_i\in L$, $i\in I$, we have
	\begin{gather}\label{eq:locvee}
	\bigg(\bigvee_{i\in I} a_i\bigg)^{\top } =\bigcap _{i\in I} a_i^{\top}.
	\end{gather}
\end{Proposition}
\begin{proof} If a poset locality relation $\top$ satisfies
		\eqref{eq:locvee}. Then from $a\top c, b\top c$, we have $c\in a^\top \cap b^\top = (a\vee b)^\top$. Hence $a\vee b$ is in $c^\top$. Thus $c^\top$ is a lattice ideal.
		
Conversely, given a \loc relation on $L$, the compatibility of the \loc relation with the partial order gives the inclusion from left to right. To show the inclusion from right to left, let~$b\in a_i^\top$ for all $i\in I$. Then $a_i\in b^\top$ for all $i\in I$ and since $b^\top$ is a lattice ideal, this implies by~induction on $i$ that $\bigvee_{i\in I} a_i \in b^\top$ so that $b\in 	 \left(\bigvee_{i\in I} a_i\right)^{\top }$.
\end{proof}

On the subspace lattice $G(V)$, Proposition~\ref{prop:locvee} translates to the following statement.
\begin{Example} %\label{pp:loccomp}
		Lattice locality relations on $G(V)$ are poset locality relations $\top_G$ such that for any index set $I$ and $W_i\in G(V)$, $i\in I$, we have
	\begin{gather*}%\label{eq:locsum}
	\bigg(\sum_{i\in I} W_i\bigg)^{\top _G} =\bigcap_{i\in I} W_i^{\top _G}.
	\end{gather*}
	Note that sum and intersection in the above equation are operations in the lattice $G(V)$.
\end{Example}

\section{Orthocomplemented lattices}\label{sec:ortholatt}
This section reviews the classical notion of orthcomplemented lattices, in preparation for the forthcoming sections in which we shall equip a subclass of orthocomplemented lattices with a~\loc relation.
See~\cite[Section~II.6 ]{B} and~\cite[Section~I.6]{Gra1} for background on complemented lattices and relatively complemented lattices. We~also refer the reader to~\cite{CM} for a study of~ortho\-complementations.

\subsection{Relatively complemented lattices}

We specialise to lattices {$(L, \leq, \wedge, \vee)$} bounded from below by $0$, equipped with the disjointedness \loc relation $\top_\wedge\colon a\top_\wedge b$ if and only if $a\wedge b=0$
so that $\wedge$ restricted to the graph of $\top_\wedge$ is identically zero.
Define
	\begin{gather*}%\label{eq:diamondloc} \notag
	a\oplus b:=a \vee b \qquad \text{when}\, {a\top_\wedge}\,b.
\end{gather*}

\begin{Definition}\label{defn:relcomplt}\qquad
\begin{enumerate}\itemsep=0pt
\item[$(a)$] A {\it complemented lattice} is a bounded lattice $(L,\leq, 0, 1)$ in which every element $a$ has a~{\it complement}, i.e., an element $a'$ in $ L$, such that $a \oplus a' = 1$.
\item[$(b)$]
A {\it sectionally complemented } lattice 	
% (as suggested by Referee 1)}
is a lattice $(L,\leq, 0)$ with bottom element $0$ in which any interval of the form $[0, b]$ is complemented when viewed as a sublattice of $L$, i.e., such that for any $a\le b$ there is an element $a^\prime$ in $ L$ such that $a \oplus a^\prime = b$.
\item[$(c)$] A {\it relatively complemented lattice} is a lattice $(L,\leq)$ in which any interval $[a, b]$ of $L$ is complemented when viewed as a sublattice of $L$. This means that any element $x\in [a, b]$ has a {\it relative complement}, namely an element $x'$ in $ L$ such that $x \vee x' = b$ and $x\wedge x'=a$.
\end{enumerate}
\end{Definition}

\begin{Remark}
Any relatively complemented lattice which has a minimal element $0$ is sectionally complemented and any sectionally complemented lattice with a greatest element is complemented. Therefore, any relatively complemented lattice with a least and a greatest element is complemented.
\end{Remark}

\begin{Example} 	The power set $\left({\mathcal P}(X), \subseteq \right)$ considered in Example~\ref{ex:powerset} is relatively complemented. Fix an interval $[A,B]$ given by $A\subseteq B$ and let $C\in [A,B]$, that is, $A\subseteq C \subseteq B$. Let $C':=A\cup (B\setminus C)$. Then $C\cup C'=B$ and $C\cap C'=A$.
	\end{Example}
We recall a well-known result:

\begin{Lemma}[{\cite[Corollary to Theorem 11.1]{Ru}}] %\label{lem:compl_mod_rel_compl}
	 Every complemented modular lattice $L$ is relatively complemented.
\end{Lemma}

\begin{Example}\label{GVrelcomplemented}
The \subsp lattice $(G(V), \preceq, \cap, +)$ for a finite dimensional vector space considered in Example~\ref{ex:GV} is relatively complemented.
As can be easily proved, for $U\subseteq W$ in $G(V)$, the interval $[U, W]$ is isomorphic to the subspace lattice of $W/U$ under the isotone assignment $X\in [U,W]$ to $X/U\in G(W/U)$. Then the existence of relative complements in $[U,W]$ follows from the existence of complements in $G(W/U)$.
\end{Example}

\begin{Example}%\label{GinfVrelcomplemented}
The closed \subsp lattice $\big(\cG(V), \preceq, \cap, +\big)$ for a Hilbert space $V$ considered in Example~\ref{ex:GV} is also relatively complemented. Indeed, for a closed subspace $U$ of $V$, a closed subspace $W$ of $U$ has a closed complement in $U$ given by its orthogonal complement space $W^\perp_U:=W^\perp\cap U$. Indeed, the latter is closed and their topological sum $W\tsum W^\perp_U=(W\oplus W^\perp)\cap U=U$.
\end{Example}

	\begin{coex}%\label{ex:divisioncompl}
The lattice $(\N, \vert)$ considered in Example~\ref{ex:powerset} is not sectionally	complemented, for example, $2\vert 4$, but there is no element $a$, such that $2\oplus a=4$.	
\end{coex}
The following example shows that the complement in a complemented lattice might not be unique.

\begin{Example}%\label{ex:GVcompl}
Back to Example~\ref{GVrelcomplemented} in the case $V=\R^2$,
and with the notations of Counter\-example~\ref{coex:GR2}, we have $W\oplus W_1= V$ and $W\oplus W_2= V$.
\end{Example}
The following lemma shows that the uniqueness of the relative complement is a strong requi\-rement.

\begin{Lemma}[{\cite[Corollary 1, p.~134]{B}}] %\label{lem:distrrelcomp}
On a lattice $(L, \leq)$, the uniqueness of the relative complement is equivalent to the distributivity property.
\end{Lemma}

\subsection{ Orthocomplemented lattices}
%\label{s:orthomodular}
We introduce the notion of orthocomplementation on a poset bounded from below, which amounts to the (strong) orthocomplementation introduced in \cite[Remark below Proposition~1.7]{CM}.

\begin{Definition} \label{defn:orthocomplementmap}
A poset $(P, \leq, 0)$ with bottom $0$ is called {\it orthocomplemented} if it can be equipped with a~map $\Psi\colon P\to P$ called the {\it orthocomplementation}, which assigns to $a\in P$ its {\it orthocomplement}\footnote{$\Psi(a)$ is often denoted by $a^\perp$ in the literature, a notation we avoid here not to cause any confusion when specialising to the case of orthogonal complements on vector spaces.} $\Psi(a)$ such that
\begin{enumerate}\itemsep=0pt
\item[$(a)$] ({\it $\Psi$ antitone})
 $b\leq a\Rightarrow\Psi(a)\leq \Psi(b)\ \forall\, (a, b)\in P^2$,
\item[$(b)$]
({\it $\Psi$ involutive})
$\Psi^2=\Id$,
\item[$(c)$] ({\it $\Psi$ \pone })
for any $b\in P$, if $b\leq a$ and $b\leq \Psi(a)$, then $b=0$.
\end{enumerate}
\end{Definition}
Here is an elementary yet useful lemma.
For $a, b$ in a poset $(P,\leq)$, let $a\vee b$ (resp.~$a\wedge b$) be the least upper bound (resp.~least lower bound) of $a$ and $b$, if it exists.

\begin{Lemma} %\label{prop:Morgan_laws}
A map $\Psi\colon P\to P$ on a poset $(P, \leq)$ which is
\begin{enumerate}\itemsep=0pt
\item[$(a)$] $(${\it antitone}$)$
			$\forall\, (a, b)\in L^2$, $b\leq a\Rightarrow\Psi(a)\leq \Psi(b)$,
\item[$(b)$]
			$(${\it involutive}$)$
			$\Psi^2=\Id$,
\end{enumerate}
		satisfies the following relations:
		\begin{gather*}%\label{eq:Psiop}
		\Psi(a \wedge b)=\Psi(a)\vee\Psi(b)\qquad {\rm and}\qquad \Psi(a \vee b)=\Psi(a )\wedge\Psi(b) \end{gather*}
for any $(a , b)\in P^2$ for which the joins arising in these identities are well-defined $($see {\rm\cite[Proposition~1.3(iii)]{CM})}.
		
		If moreover, $P$ has a bottom $0$, then $P$ is bounded with top $\Psi(0)$
		$($see~{\rm\cite[Proposition~1.3(iii)]{CM})}.
\end{Lemma}

\begin{Definition} \label{defn:orthocomplemented_lattice}	
A lattice $(L, \leq, \wedge, \vee)$ bounded from below by $0$ equipped with a map $\Psi\colon L\to L$ which defines an orthocomplementation on the poset $(L, \leq , 0)$, is called an {\it orthocomplemented lattice}.
\end{Definition}

Thanks to the above lemma, the separating condition in Definition~\ref{defn:orthocomplementmap} can be replaced by an a priori stronger condition.

\begin{Proposition}[{\cite[Propositions 1.1(iii) and~1.7]{CM}}] \label{prop:Psistronglysep}\qquad
\begin{itemize}\itemsep=0pt
	\item An antitone and involutive map $\Psi$ on a lattice $(L, \leq, \wedge, \vee)$ induces morphisms
 $\Psi\colon (L, \wedge)\to (L, \vee)$ and $\Psi\colon (L, \vee)\to (L, \wedge)$ on $L$.
 \item
A lattice $L$ bounded from below equipped with an orthocomplementation $\Psi$ is bounded from above with top $1=\Psi(0)$ and satisfies the following {\it strongly separating} property 	 \begin{gather}\label{eq:Psistronglyseparating}
(\Psi \text{ \it strongly \pone})\qquad a\oplus \Psi(a) =1 \qquad \forall\, a\in L.
\end{gather}
Thus an orthocomplemented lattice enjoys the strongly separating property.
\end{itemize}
\end{Proposition}

Here is a class of examples of orthocomplemented lattices of direct interest to us.
\begin{Example}%\label{ex:orthogonalcomplement}
This is a classical example, see, e.g., \cite[below Definition 1.4]{CM}.
	
	Given a Hilbert vector space $(V, \langle\cdot, \cdot\rangle)$,
	and $\cG(V)$ the closed \subsp lattice of Example~\ref{ex:GV}. The map
	\begin{gather*}%\label{eq:Psiperp}
	\Psi_{\langle\cdot, \cdot\rangle}\colon \ \cG(V) \longrightarrow \cG(V),\qquad
	W\longmapsto W^\perp:=\{v\in V\,|\, \langle v, w\rangle=0,\, \forall\, w\in W\}
	\end{gather*}
	defines an orthocomplementation on $\cG(V)$.
	
	Indeed, $W^\perp$ is closed (by the continuity of the inner product) for any linear (whether closed or not) subspace $W\preceq V$ and for any closed linear subspaces $W$, $W_1$, $W_2$ of $V$, we have
	\begin{enumerate}\itemsep=0pt
	\item[$(a)$] ({\it $\Psi$ \pone}) {$W\oplus W^\perp =\overline W\oplus \overline{W^\perp}=V$},
	\item[$(b)$] ({\it $\Psi$ antitone}) if $W_1\leq W_2$, then $W_2^\perp\preceq W_1^\perp$,
	\item[$(c)$]
		({\it $\Psi$ involutive}) ${W^\perp}^\perp=\overline W=W$.
	\end{enumerate}
\end{Example}
\begin{Corollary}\label{coro:PsiR2}
% Strongly separating \zb {no need of "strongly separating"}
Orthocomplementations on $G\big(\R^2\big)$ are in one-one correspondence with involutive maps $\psi\colon [0, \pi)\to [0, \pi)$ without fixed points.
\end{Corollary}
\begin{proof}
An orthocomplementation $\Psi\colon G\big(\R^2\big)\rightarrow G\big(\R^2\big)$ obeys the strongly separating condition~\eqref{eq:Psistronglyseparating}
\begin{gather*}
U\oplus \Psi(U) =\R^2 \qquad \forall\, U\in G\big(\R^2\big).
\end{gather*}
	In particular, $\Psi (\{0\})=\R^2$ and $\Psi\big(\R^2\big)=\{0\}$. Thus $\Psi$ is uniquely determined by its effect on the lines $\R {\rm e}^{\theta {\rm i}}$ in bijection with $\theta\in [0,\pi)$. Let $\psi\colon [0,\pi)\to [0,\pi)$ be defined by $\Psi(\R {\rm e}^{\theta {\rm i}})={\rm e}^{\psi(\theta){\rm i}}$. Then~$\Psi$ being involutive (resp.~strongly separating) amounts to $\psi$ being involutive (resp.~without fixed points). The antitonicity of $\Psi$ is trivial.
\end{proof}

\section{\Pone \loc relations and orthocomplementations}
%\label {s:locality}
This section is dedicated to a class of \loc relations on bounded lattices from which we build orthocomplementations. We~establish a one-one correspondence
\begin{gather*}
\{\text{orthocomplementations}\} \quad \longleftrightarrow\quad \{\text{strongly \pone \loc relations}\}
\end{gather*}
for bounded lattices. As we learned from referee reports on a previous version of this paper, this one-one correspondence is actually known in lattice theory (see, e.g.,~\cite{CM}). We~nevertheless use a~slightly different terminology which we find well-suited for the locality setup we have in mind.

\subsection {\Pone \loc relations}

A finite sublattice $M=\{a_1,\dots,a_N\}$ of a lattice admits a greatest element $a_1\vee\cdots\vee a_N$. In~general a sublattice $M$ of $L$ does not have a greatest element, even if $(L,\leq,\wedge, \vee)$ is bounded.

\begin{coex} %\label{rk:sublattices}
	As a
	counterexample, take $L=\N \cup \{\infty \}$
	with $\leq$ the usual order on natural number and $\infty \ge n,$ for any $n\in \N$.
	Set $n\vee m:=\max(n,m)$ and $n\wedge m:=\min(n,m)$. Then $(\N,\leq,\wedge, \vee)$ is a sublattice of $(L,\leq,\wedge, \vee)$, but does not admit
	a greatest element.
\end{coex}
The following technical lemma will be useful for what follows.
\begin{Lemma} \label{lem:maxtop}
	Let $S $ be a subset of a locality poset $(P,\leq,\top)$ $($Definition~$\ref{de:locposet})$ admitting a~greatest element $\grt(S)$. Then
	\begin{gather}\label{eq:maxtop}
	\grt (S)^\top= S^\top.
	\end{gather}
\end{Lemma}
\begin{proof}
Clearly, $S^\top \subseteq \grt (S)^\top$. On the other hand, for any $a\in S$, from $a\leq \grt (S)$ we obtain $\grt (S)^\top \subseteq a^\top$. Hence $\grt (S)^\top \subseteq \cap_{a\in S} a^\top =S^\top$.
\end{proof}

The subsequent definition relates to that of
	weak and strong orthogonality introduced in~\cite[Definition 2.2]{CM}. To compare our separating property with weak orthogonality, the following elementary lemma is useful.

\begin{Lemma}\label{lem:nondegseparating}
In a \loc lattice $(L, \leq, \top, 0, 1)$, the {\it non-degeneracy condition} \begin{gather*}%\label{eq:atopa}
a\top a\Longrightarrow a=0\qquad \forall\, a\in L
\end{gather*}
$($written $\ker \top =0$ in {\rm\cite{CM}}$)$ is equivalent to the condition
\begin{gather*}%\label{eq:atopab}
a\top b\Longrightarrow a\wedge b=0\qquad \forall\, (a, b)\in L^2.
\end{gather*}
\end{Lemma}

\begin{proof}
Taking $b=a$ in the second condition gives the non-degeneracy condition. Conversely, from $a\top b$ we have $(a\wedge b)\top (a\wedge b)$. Then the non-degeneracy condition gives $a\wedge b=0$.
\end{proof}

By Proposition-Definition~\ref{de:locposet}, on a poset with locality $(P, \leq)$ we have $\downarrow a\subseteq \big(a^\top\big)^\top $. Imposing the converse inclusion leads to the following notion.

\begin{Definition}\label{def:splitlattice}
	A \loc relation $\top$ on a lattice $(L, \leq, 0, 1) $
	is called {\it \pone} if the following conditions hold.
	\begin{enumerate}\itemsep=0pt
		\item[$(a)$] %\label{de:proper_ii}
	 $a\top b\Rightarrow a\wedge b=0$ for any $a$ and $b$ in $L$,
		\item[$(b)$] the set $a^\top$ admits a greatest element $\grt\big(a^\top\big)$ for any $a$ in $L$.
	\end{enumerate}
In this case, we say that $(L, \leq 0, 1, \top) $ is a {\it \ptwob \loc lattice}.
\end{Definition}

\begin{Remark}
It follows from Lemma~\ref{lem:nondegseparating}, that a weak (non-degenerate) orthogonality lattice of~\cite[Definition 2.2]{CM} satisfies the separating property.
\end{Remark}

Assumption (b) on the existence of $\grt\big(a^\top\big)$ which corresponds to completeness in \cite[Definition 2.4]{CM} is rather strong.
	\begin{coex}The diamond lattice	$L=\{0,a,b,c,1\}$ of Example~\ref{ex:diamond}, endowed with the \loc relation $x\top y\Leftrightarrow x\wedge y=0$, does not satisfy this assumption, since $a^\top=\{0,b,c\}$, $b^\top=\{0,a,c\}$ and $c^\top=\{0,a,b\}$.
\end{coex}

\begin{Proposition}	\label{lem:full}
	In a separated locality lattice $(L, \leq,\top, 0,1 )$, the following conditions are equi\-valent
\begin{gather} 	
a^\top=\{0\}\Leftrightarrow a=1 \qquad \forall\, a\in L, \label{eq:extreme}
\\
a\oplus \grt\big(a^\top\big)=1\qquad \forall\, a\in L\qquad (\text{\it closedness condition}).
\label{eq:closedness}
\end{gather}
Assuming~\eqref{eq:closedness} holds, we have
\begin{gather} 	
\label{eq:closednessweak}
a^\top=L\Leftrightarrow a=0\qquad \forall\, a\in L.
\end{gather}
\end{Proposition}

\begin{proof}
We first prove~\eqref{eq:extreme} $\Longrightarrow$~\eqref{eq:closedness}.
Since $\big(a\oplus \grt \big(a^\top\big)\big)^\top \subseteq a^\top \cap \big(\grt \big(a^\top\big)\big)^\top = a^\top \cap \big(a^\top\big)^\top$, for $c\in (a\oplus \grt (a^\top))^\top$ we have $c\top c$ and hence $c\wedge c=c=0$ by $(a)$ in Definition~\ref{def:splitlattice}. Thus $\big(a\oplus \grt \big(a^\top\big)\big)^\top=\{0\}$ and $a\oplus \grt \big(a^\top\big)=1$ by assumption.

We next prove~\eqref{eq:extreme} $\Longleftarrow$ (\ref{eq:closedness}).
From $a^\top = \{0\}$ we have
$1=a\oplus \grt\big(a^\top\big)=a\oplus 0=a$. From $a=1$, we have
$0=a\wedge \grt \big(a^\top\big) =1\wedge \grt\big(a^\top\big)=\grt\big(a^\top\big)$. Then $a^\top=0$.

To prove (\ref{eq:closedness}) $\Longrightarrow$~\eqref{eq:closednessweak}, let us first notice that the implication $a^\top=L\Rightarrow a=0$ holds as a~consequence of $a\top a$, which in turn implies $a\wedge a=a=0$ by $(a)$ in Definition~\ref{def:splitlattice}. It therefore suffices to show that Assumption~\eqref{eq:closedness} implies $0^\top=L$. This follows from $0\oplus\grt\big(0^\top\big)=1\Rightarrow\grt\big(0^\top\big)=1$. Thus $1\in 0^\top$. Since $0^\top$ is a poset ideal by Definition~\ref{def:loclattice} and since $b\leq 1$ $\forall\, b\in L$, we have $b\in 0^\top $ $\forall\, b\in L$.
\end{proof}

An easy counterexample shows that~\eqref{eq:closednessweak} does not imply~\eqref{eq:closedness}.

\begin{coex}We equip the bounded lattice $L=\{0,a,b,1\}$ defined by the partial order $0\leq a\leq 1$ and $0\leq b\leq 1$ with the locality relation $\top$ defined by $0^\top=L$, $a^\top=b^\top=1^\top$. Then clearly~\eqref{eq:closednessweak} holds but~\eqref{eq:closedness} does not.\end{coex}
\begin{Definition}
A separating locality relation $\top$ on a lattice $L$ is called {\it strongly \pone} if it satisfies
	\begin{itemize}\itemsep=0pt	
\item [$(c)$]%\label{de:proper_iii}
		$\grt\big(\big(a^\top\big)^\top\big)= a $ for any $a\in L$,
\item [$(c')$]%\label{de:proper_iv}
or equivalently $\downarrow a\supseteq \big(a^\top\big)^\top $,
	\end{itemize}in which case the lattice endowed with the \loc relation is called {\it strongly \ptwob}. We~also say that $(L, \leq, 0, 1, \top) $ is a strongly separated locality lattice.
	\label{def:strongsplitlattice}
\end{Definition}
Not all separating lattices are strongly separating as the next example shows.
\begin{Example}
		Let $\R^3$ be equipped with the canonical basis $\{e_1,e_2,e_3\}$. We~consider the locality relation $\top$ on $\R^3$ defined by $0_{\R^3}\top\R^3$, $e_i\top e_{i+1}$ for $i\in \{1,2 \}$ extended by symmetry and linearity. By construction, $\big(\R^3,\top\big)$ is a locality vector space and $\top$ endows $G\big(\R^3\big)$ with a lattice locality structure.
		
In $G\big(\R^3\big)$, the only non trivial polar sets are $0^\top = G\big(\R^3\big)$,
$\langle e_1\rangle^\top = \{0,\langle e_2 \rangle\}$,
$\langle e_2\rangle^\top = \{0,\langle e_1\rangle,\langle e_3\rangle,\langle e_1, e_3\rangle\}$,
$\langle e_3\rangle^\top = \{0,\langle e_2 \rangle\}$,
$\langle e_1, e_3\rangle^\top =\{0, \langle e_2\rangle\}$. In~particularly $\top$ is a separating loca\-lity relation on $G(V)$. However it is not strongly separating since
\begin{gather*}
\grt\big(\big(\langle e_1\rangle^\top\big)^\top\big) = \grt\big(\{0, \langle e_2\rangle\}^\top\big)= \langle e_1, e_3\rangle.
\end{gather*}
\end{Example}

 \begin{Remark}\qquad
 \begin{enumerate}\itemsep=0pt
 		\item[$(a)$] Under the additional assumption $(b)$ of a separating locality, the mere strong separation property $(c)$ actually amounts to the strong orthogonality of \cite[Definition~2.3]{CM}.
 		\item[$(b)$] Note that the existence of $\grt\big(\big(a^\top\big)^\top\big)$ in $(c)$ follows from the existence of $\grt\big(a^\top\big)$ for any $a\in L$. Indeed, we can then apply~\eqref{eq:maxtop} $\grt(S)^{\top}=S^{\top}$ to $S:=a^\top$, which yields $\grt\big(a^\top\big)^\top= \big(a^\top\big)^\top$.
 \end{enumerate}
 \end{Remark}

\begin{Example}
	The \loc relation in Example~\ref{ex:locposetex} is strongly \pone. It is clearly \pone with $\grt\big(A^\top\big)=X\backslash A$.
	Then by Lemma~\ref{lem:maxtop},
	\begin{gather*}
\grt \big(\big(A^\top\big)^\top\big)=\grt\big( \grt\big(A^\top\big)^\top\big)=X\backslash (X\backslash A)=A.
\end{gather*}
\end{Example}
\begin{Example}
The \loc lattice $\big(\cG(V), \preceq, \{0\}, V, \top \big)$ in Example~\ref{ex:GVlocality}, is strongly \ptwob by the same argument, noting that for any $U\in \overline G(V)$, the linear space $\grt\big(U^\top\big)$ is the orthogonal complement of $U$ for the inner product.
\end {Example}

\begin{Corollary} %\label{cor:full}
A strongly separated locality lattice $(L, \leq,\top,0,1)$ satisfies the closedness condition~\eqref{eq:closedness}.
\end{Corollary}
\begin{proof} We show that a strongly separated locality lattices satisfies~\eqref{eq:extreme}.
To show the implication from right to left, if $b\in 1^\top$, then $b\top 1$, so $b\top b$ which means $b=0$, so $1^\top =\{0\}$. The converse implication follows from $(c)$ in Definition~\ref{def:strongsplitlattice}. Indeed, since $0^\top=L$, we have $a^\top=\{0\}\Rightarrow \big(a^\top\big)^\top=L$ which implies that $a=\grt L=1$.
\end{proof}

The subsequent example shows that the \pone property is not hereditary.
	\begin{Example} Consider the lattice
	$L=\{0,a,b,1\}$, $0\leq a\leq 1$ and $0\leq b\leq 1$ and no relation between $a$ and $b$. On $L$ endowed with the disjointedness \loc relation $\top:=\top_\wedge$, we have $0^\top=L$, $1^\top =0$, $a^\top=\{0,b\}$ and $b^\top=\{0,a\}$. So $L$ is (strongly) \pone, but not
	its restriction to $\widetilde L=\{0,a,1\}$, since $a^\top =\{0\}$.
\end{Example}

\subsection {One-one correspondence}
	 It turns out to be a specialisation of \cite[Theorem 3.1]{CM}, which relates weak degenerate orthogonalities and weak degenerate orthocomplementations on posets, to a one-one correspondence between strong (non degenerate) orthogonalities and strong (non degenerate) orthocomplementations on bounded lattices. The latter is what we need in view of the applications we have in~mind.

Proposition~\ref{lem:full} provides sufficient conditions on a locality poset for the existence of a candidate orthocomplement to any element $a$, namely $\grt \big(a^\top\big)$ arising in (\ref{eq:closedness}). The subsequent statement confirms and strenghtens this fact.

\begin{Theorem} \label{thm:complement_strongly_local}
Let $(L, \leq, 0, 1)$ be a bounded lattice.
\begin{enumerate}\itemsep=0pt
	\item Let $\top$ be a \loc relation on the set $L$ for which $(L, \leq, 0,1, \top) $ is a lattice with strongly \ptwob \loc. The assignment
	\begin{gather*}%\label{eq:Psitop}
	\Psi:=\Psi^\top\colon\quad L\to L,\qquad a\mapsto \grt \big(a^\top\big)
	\end{gather*}
	is an orthocomplementation.
	%\label{it:split1}
	\item Conversely, given an orthocomplementation $\Psi $ on $L$, the \loc relation $\top:=\top_\Psi$ defi\-ned~by
	\begin{gather*}%\label{eq:PsitoTop}
a\top b\Longleftrightarrow b \leq \Psi(a)
\end{gather*}
	yields a strongly \pone \loc relation.
	%\label{it:split2}
	\item The maps defined this way:
	\begin{gather*}
F\colon\ \{\text{strongly \pone \loc relations on L}\}\to
\{\text{orthocomplementations on L}\}
\end{gather*}
	and
	\begin{gather*}
 G\colon\ \{\text{orthocomplementations on L}\} \to
 \{\text{strongly \pone \loc relations on L}\}
 \end{gather*}
	are inverse to each other.
	%\label{it:split3}
\end{enumerate}
%\label{thm:splitcompl}
\end{Theorem}

\begin{proof}
$(a)$
It follows from Proposition~\ref{lem:full}, that on a lattice $(L, \top)$ with strongly \ptwob \loc, the map $a\mapsto \Psi^\top(a):=\grt (a^\top)$, which is well-defined, satisfies $a\oplus \Psi^\top(a)= 1$. This is~\eqref{eq:Psistronglyseparating} in Proposition~\ref{prop:Psistronglysep}. To show that $\Psi$ is an orthocomplementation, we need to check that it is involutive and isotone. From $a\leq b$ we have $b^\top\subseteq a^\top$ which yields $\grt \big(b^\top\big)\leq\grt \big(a^\top\big)$. So~$\Psi^\top$ is antitone. Further we have
\begin{gather*}
\big(\Psi^\top(a)\big)^\top=\grt\big(a^\top\big)^\top=\big(a^\top\big)^\top
\end{gather*}
as a consequence of~\eqref{eq:maxtop} applied to $S=a^\top$. So
\begin{gather*}
\Psi^\top\big(\Psi^\top(a)\big)=\grt\big(\Psi^\top(a)\big)^\top=\grt \big(a^\top\big)^\top.
\end{gather*}

The strongly separating condition $(c')$ in Definition~\ref{def:strongsplitlattice} tells us that
$\grt\big(\big(a^\top\big)^\top\big)= a$ so that
$\Psi^\top\big(\Psi^\top(a)\big)=a$, which ends the proof of the involutivity.

$(b)$ %\eqref{it:split2}
Let $\Psi\colon L\to L$ be an orthocomplementation. We~show that
\begin{gather*}
a\top b\Longleftrightarrow b\leq \Psi(a)
\end{gather*}
is a strongly \pone \loc relation on $L$ in several steps. Note that $\top$ is equivalently defined by $a^\top:=\downarrow\Psi(a)$ for all $a\in L$.
\begin{enumerate}\itemsep=0pt
	\item[(1)] $\top$ is symmetric since the monotonicity and involutivity of $\Psi$ yield that if $b\leq \Psi(a)$, then $ \Psi(b)\geq \Psi(\Psi(a))=a$.
	\item[(2)] $\top$ equips the poset $(L, \leq)$ with a \loc lattice structure (see Definition~\ref{def:loclattice}) since $a^\top:=\downarrow\Psi(a)$ is a (principal) lattice ideal (see Definition~\ref{de:latticenotion}).
	\item[(3)] We have $0^\top=\downarrow \Psi(0)=\downarrow 1=L$ and $1^\top =\downarrow \Psi(1)=\downarrow 0=\{0\}$.
	\item[(4)] By definition, $a\wedge \Psi (a)=0$. Now if $a\top b$, then $b\le \Psi (a)$, so $a \wedge b=0$. Furthermore, $a^\top = \downarrow \Psi(a)$ means that $\grt\big(a^\top\big)=\Psi(a)$. Hence $(L,\leq, 0, 1,\top)$ is a lattice with separated \loc.
	\item[(5)] Finally,
\begin{gather*}
\grt\big(\big(a^\top\big)^\top\big) = \grt\big( (\downarrow \Psi(a))^\top\big) =\grt\big( \Psi(a)^\top\big) =\grt \big( \downarrow \Psi^2(a)\big) = \grt (\downarrow a)=a.
\end{gather*}
This proves that $\top$ is strongly separating.
\end{enumerate}

$(c)$ For a lattice $L$ strongly \pone \loc relation $\top $, by definition
\begin{gather*}
(a,b)\in G(F(\top))\Leftrightarrow b\le F(\top)(a)\Leftrightarrow b \le \grt \big(a^\top\big)\Leftrightarrow (a,b)\in \top.
 \end{gather*}
So
\begin{gather*}
G(F(\top))=\top.
\end{gather*}
For an orthocomplementation $\Psi $ on $L$,
\begin{gather*}
F(G(\Psi ))(a)=\grt \big(a^{G(\Psi)}\big)=\Psi (a).
\end{gather*}
Hence
\begin{gather*}
F(G(\Psi))=\Psi.\tag*{\qed}
\end{gather*}
\renewcommand{\qed}{}
\end{proof}

Combining Corollary~\ref{coro:PsiR2} with Theorem~\ref{thm:complement_strongly_local} yields the following characterisation of strongly locality relations on $G\big(\R^2\big)$.
	\begin{Corollary}\label{cor:topGR2} Strongly separating locality relations on $\R^2$ are in one-one correspondence with involutive maps $\psi\colon [0, \pi)\to [0, \pi)$ without fixed points.
		\end{Corollary}
	
\section{Locality relations on vector spaces}\label{sect:locvectspaces}
In the previous section, we established a one-one correspondence between the set of
strongly \pone \loc relations on a bounded lattice and the set of orthocomplementations on the lattice. In~this section, we apply the correspondence to the subspace lattice of a finite dimensional vector space, a case of direct interest for the applications to renormalisation we have in mind, as explained in the introduction.

\subsection{Correspondence between locality relations}

From a set locality $\top _V$ on a vector space $V$, we build a poset locality relation $\top _G:=\top_{V,G}$ on~the lattice $\big(G(V), \preceq\big)$:
\begin{gather}
\label{eq:locVtoGV}
U\top _{V,G} W\qquad \text{\rm if } u\top _V w \qquad \forall\, u\in U,\ w\in W.
\end{gather}
	On the other hand,
	a set locality relation $\top_G$ on $G(V)$ (ignoring its lattice structure) induces a~locality relation $\top _V:=\top _{G,V}$ on $V$ by
\begin{gather} \label{eq:locGVtoV}
v_1\top_{G,V} v_2\qquad \text{if } \K v_1 \top _G \K v_2 \qquad\forall\, v_1,\, v_2\in V.
\end{gather} 	

As in~\cite{CGPZ1}, we call {\it locality vector space}, a vector space $V$ equipped with a (set) \loc rela\-tion~$\top$ such that the polar set $X^\top$ of any subset $X\subseteq V$ is a vector subspace of~$V$.

\begin{Proposition} \label{pp:lvtoll}
Let $\LR(V)$ denote the set of vector space locality relations $\top_V$ on $V$ and let $\LGR(V)$ denote the set of lattice locality relations $\top_G$ $($see Definition~$\ref{def:loclattice})$ on the set $G(V)$ such that
 $\{0\}^{\top _G}=G(V)$.
The equations~\eqref{eq:locVtoGV} and~\eqref{eq:locGVtoV} gives a one-one correspondence between $\LR(V)$ and $\LGR(V)$.
\end{Proposition}

\begin {proof}
Let $\top_V$ be a vector locality relation on $V$. Then for any $W\in G(V)$, $W^{\top_V}$ is a subspace and hence contains $0$. Thus we have $\{0\}^{\top_{V,G}}=G(V)$.
Clearly, for any $U\in G(V)$,
\begin{gather*}
U^{\top _V}\top _{V,G} \ U
\end{gather*}
and, for $W\in G(V)$, if $W\top _{V,G} \ U$, then $W\subseteq U^{\top _V}$. Therefore,
\begin{gather}\label{eq:relvr}
 U^{\top_{V,G}} = \downarrow \big(U^{\top_V}\big),
 \end{gather}
which is a lattice ideal. Hence $\top_{V,G}$ is in $\LGR(V)$.

Conversely, give a locality relation $\top_G\in \LGR(V)$, then $\{0\}^{\top _{G}}=G(V)$. So for any $v\in V$, $\{0\}\top _{G} \K v$, yielding $0\top_{G,V} v$.
Next let $X \subseteq V$. Then the lattice ideal $(\K X)^{\top_G}$ has a greatest element $\grt (\K X)^{\top_G}=\sum _{W\top _G \K X}W $. If $W\top _G \K X$, then $W\subseteq X^{\top_{G,V}}$, that is $\grt (\K X)^{\top_G} \subseteq X^{\top_{G,V}}$. On the other hand, if $y\in X^{\top_{G,V}}$, then $\K y \top _G \K X$, so $y\in\K y\subseteq \grt (\K X)^{\top_G}$. Thus
\begin{gather} \label{eq:relrv}
X^{\top_{G,V}}= \grt (\K X)^{\top_G}
\end{gather}
is a subspace. That is, ${\top_{G,V}}\in \LGR (V)$.

Thus we obtain maps
\begin{gather*}
\phi\colon\quad\LR(V)\longrightarrow \LGR(V),\qquad \top_V\mapsto \top_{V,G},
\\
\psi\colon\quad\LGR(V)\longrightarrow \LR(V),\qquad \top_G\mapsto \top_{G,V}.
\end{gather*}

We set
\begin{gather*}
\top_{V,G,V}:=\psi(\phi(\top_V))=\psi(\top_{V,G}).
\end{gather*}
Then for $X\subseteq V$, by~\eqref{eq:relvr} and~\eqref{eq:relrv} we have
\begin{gather*}
X^{\top_{V,G,V}}= \grt \big( (\K X)^{\top_{V,G}}\big) =\grt \big (\!\downarrow (\K X)^{\top_V}\big)
=(\K X)^{\top_V} = X^{\top_V}.
\end{gather*}

Similarly, let us set
\begin{gather*}
\top_{G,V,G}:=\phi(\psi(\top_G)) = \psi(\top_{G,V}).
\end{gather*}
For a principal poset ideal $I=\downarrow a$, we have $\downarrow \grt (I) = \downarrow \grt (\downarrow a) = \downarrow a=I$.
Thus for $U\in G(V)$, by~\eqref{eq:relvr} and~\eqref{eq:relrv} we obtain
\begin{gather*}
U^{\top_{G,V,G}} = \downarrow \big(U^{\top_{G,V}}\big) = \downarrow \big (\grt \big(U^{\top_G}\big)\big) =U^{\top_G}.
\end{gather*}

We have proved that the maps $\phi$ and $\psi$ are mutual inverses. This proves the proposition.
\end{proof}

\subsection {Non-degenerate locality relations on vector spaces}
\begin{Definition}
\label{de:vpproper}
A \loc relation $\top$ on a vector space $V$ is called \emph{non-degenerate} if $v\top v \Rightarrow v=0$ for any $v\in V$,
it is called \emph{strongly non-degenerate} if moreover for any subspace $ U \subsetneqq
V$, the polar space $U^\top $ is nonzero.
\end{Definition}

\begin {lemma}%\label{lem:strongnondeg}
In a strongly non-degenerate locality vector space $(V, \top _V)$, for any $U\leqq V$,
\begin{gather*}%\label{eq:strongnondeg}
V=U\oplus U^{\top _V}.
\end{gather*}
\end{lemma}

\begin {proof}
By strong non-degeneracy, we have
\begin{gather*}
V=U+U^{\top _V},
\end{gather*}
otherwise we can find $0\neq v\in \big(U+U^{\top _V}\big)^{\top _V}$ which means $v\in U^{\top _V}$ and $v\in \big(U^{\top_V}\big)^{\top_V}$, so $v\top _V v$. Then $v=0$, a contradiction. Further, the non-degeneracy gives
\begin{gather*}
U\cap U^{\top _V}=\{0\},\end{gather*}
yielding the conclusion.
\end {proof}

\begin {Proposition}
\label {prop:VGV2}
A locality $\top _V$ on a vector space $V$ is
\begin{enumerate}
\item %\label{it:vgv2a}
non-degenerate if and only if $(G(V),\top_{V,G})$ is a \loc lattice which has the separating property:
$U \top_{V,G}\ W\Rightarrow U\cap W=0$,
\item %\label{it:vgv2b}
strongly non-degenerate if and only if $(G(V),\top_{V,G})$ is a strongly separating \loc lattice.
\end{enumerate}
In this case, the map $U\mapsto U^{\top_V}$ defines a orthocomplementation on $G(V)$.
\end{Proposition}
	
\begin {proof}
$(a)$
Let $(V,\top_V)$ be a non-degenerate \loc vector space. Then by Proposition~\ref{pp:lvtoll}, $\top_{V,G}$ equips $G(V)$ with a \loc lattice. If $U\top_{V,G} W$, then for $v\in U\cap W$, $v\top _V v$, so $v=0.$ Thus $U \cap W =\{0\}$.

Now suppose that $(G(V), \top_{V,G})$ is a \loc lattice with the separating property. We~already know $(V, \top _V)$ is a locality vector space. If $v\top _V v$, then $\K v\top _{V,G} \K v$. So from the separating condition we have $\K v=\K v\cap \K v=\{0\}$. So $v=0$, which means $(V, \top_V)$ is non-degenerate.

$(b)$
If $(V, \top _V )$ is a strongly non-degenerate locality vector space, then
for any $U\subseteq V$, by~\eqref{eq:relrv}, $\grt\big(U^{\top _{V,G}}\big)=U^{\top _V}$.
Further, we have $U\in \big(U^{\top _{V,G}}\big)^{\top _{V,G}}$ by definition. By non-degeneracy of $\top _V$, $V=U\oplus U^{\top _V}$ and $U=\big(U^{\top _V}\big)^{\top _V}$. Thus
\begin{gather*}
\big(U^{\top _{V,G}}\big)^{\top _{V,G}}=\big(\grt\big(U^{\top _{V,G}}\big)\big)^{\top _{V,G}}=\big(U^{\top _V}\big)^{\top _{V,G}}
\end{gather*}
and hence
\begin{gather*}
\grt \big(U^{\top _{V,G}}\big)^{\top _{V,G}}=\grt \big(U^{\top _V}\big)^{\top _{V,G}}=\big(U^{\top _V}\big)^{\top _V}=U.
\end{gather*}

Now assume that $(G(V), \top _{V,G})$ is a strongly separating \loc lattice. If $U\subsetneqq V$, then $U^{\top _{V,G}} \not =\{0\}$, otherwise $\big(U^{\top _{V,G}}\big)^{\top _{V,G}}=G(V)$, we have $U=V$ by taking the greatest element. Now take $0\neq w\in U^{\top _{V,G}}$. Then $\K w\leq U^{\top _{V,G}}$. Hence $\K w \top _{V,G} U$, and $\K w \top _{V,G} \K u$ for any $u\in U$, that is, $w\in U^{\top _V}$. Thus $U^{\top_V}\neq \{0\}$.

The last assertion is a consequence of Theorem~\ref{thm:complement_strongly_local}.
\end{proof}

\subsection{Locality bases}
By Proposition \ref{prop:VGV2}, we can construct locality relations on subspace lattice from locality relations on the underlying vector space so that we now focus on vector spaces. In~much the same way as a Euclidean vector space can be equipped with an orthogonal basis, we show that a vector space with a strongly
\pone \loc relation possesses a basis that is compatible with the relation. However, the relation is not uniquely determined by this basis.

\begin{Definition} \label{defn:compl_map_TVS}
For locality vector space $(V, \top)$, a basis ${\mathcal B}=\{e_\alpha \}_{\alpha \in \Gamma} $ of $V$ is called a {\it \loc basis} for $\top$ if the basis vectors are mutually independent for $\top$, i.e.,
\begin{gather*}%\label{eq:localitybasis}
\text{ if } \alpha \neq \beta \qquad \text{ then } e_\alpha \top e_\beta.
\end{gather*}
\end{Definition}

We now study \loc relations on a vector space $V$ with a countable basis. We~start with a generalisation to arbitrary strongly separating locality relations (instead of orthogonality) of \cite[Chapter II, Theorem 1]{Gross79}.
\begin{Proposition}\label{prop:Psisplitting}	
Let $(V, \top) $ be a strongly \pone \loc vector space of countable dimension.
%with a countable basis, and $\top$ be a strongly \pone \loc relation on $V$,
Then $V$ has a \loc basis for $\top$.
\end{Proposition}

\begin{proof} This process is similar to the Gram--Schmidt process in linear algebra. By assumption, $V$ admits a countable basis ${\mathcal B}=\{e_n\,|\, n\in \N\}$. We~apply the induction on $n$ to construct a~locality basis $\{u_1,\dots,u_n\}$ for the subspace $W_n$ of $V$ spanned by $\{e_1,\dots,e_n\}$.

First take $u_1=e_1$ and then $\K\{u_1\}=\K\{e_1\}$. Assume that a \loc basis $\{u_1,\dots,u_n\}$ of~$W_n:=\K\{e_1,\dots, e_n\}$ has been constructed.
Let~$\Psi^\top$ be the polar map
induced by $\top$: $\Psi^\top(W):=\grt\big(W^\top\big)$.
By Example~\ref{ex:GVmodularity}, $G(V)$ is modular. So we have
$W_{n+1}=W_n\oplus \big(\Psi^\top(W_n)\cap W_{n+1}\big)$. Let $u_{n+1}$ be a nonzero element in $\Psi^\top(W_n)$. Then $\{u_1,\dots,u_{n+1}\}$ is a locality basis of~$W_{n+1}$. This completes the induction. Then $\{u_i\}_{i\geq 1}$ is a \loc basis of $V$.
\end{proof}

We now provide an example that shows that, unlike a basis which uniquely determines a~vector space, a~locality basis is not enough to determine the locality vector space, suggesting the richness of locality relations on a vector space.

\begin{Remark}
	A strongly non-degenerate locality relation on $\R^2\simeq \C$ therefore has infinitely many locality bases $\big\{{\rm e}^{\theta {\rm i}}, {\rm e}^{\psi(\theta){\rm i}}\big\}$ parametrised by $\theta\in [0, \pi)$.
\end{Remark}
\begin{proof}By Proposition~\ref{prop:VGV2}, strongly non-degenerate locality relations on $\R^2$ are uniquely deter\-mi\-ned by strongly separating relations on $G\big(\R^2\big)$. By Corollary~\ref{cor:topGR2}, these in turn are in one-one correspondence with involutive maps $\psi\colon [0,\pi)\to [0, \pi)$ without fixed points, which give a~strongly separating orthocomplement $\Psi\colon \R {\rm e}^{\theta {\rm i}}\mapsto \R {\rm e}^{\psi(\theta){\rm i}}$. Thus strongly non-degenerate loca\-lity relations on $\R^2$ are determined by involutive maps $\psi\colon [0,\pi)\to [0, \pi)$ without fixed points. \end{proof}

\appendix
\section{An alternative correspondence}

Following the suggestion by a referee, we present here another class of symmetric binary relations which can be put in one-one correspondence with a weak form of orthocomplementations. This is carried out on the class of complete atomistic lattices to give a one-one correspondence between a class of locality lattice relations and certain locality relations on the set of atoms of $L$.
\begin{Definition}Let $L$ be a lattice bounded from below by $0$.
	\begin{enumerate}\itemsep=0pt
		\item An {\it atom} in $L$ is an element $p$ in $L$ that is minimal among the non-zero elements, i.e., such that for $x\in L$, $x< p$ if and only if $x=0$.
		\item $L$ is an {\it atomic} lattice whose elements $a$ are all bounded from below by some atom, i.e., there is an atom $p$ in $L$ such that $p\leq a$.
		\item
		A lattice is {\it atomistic} if it is atomic and every element is a join of some finite set of atoms.
	\end{enumerate}
\end{Definition}
\begin{Example}
	The lattice of divisors of $4$, with the partial ordering ``is divisor of", is atomic, with $2$ being the only atom. It is not atomistic, since $4$ cannot be obtained as least common multiple of atoms.
\end{Example}
\begin{Definition} A lattice is {\it complete} if all subsets have both a supremum (join) and an infimum (meet).
\end{Definition}\begin{Example} Every non-empty finite lattice is complete.
\end{Example}

\begin{Example}Let $V$ be a finite dimensional vector space.
	The subspace lattice $G(V)$ con\-si\-de\-red previously is an atomistic complete lattice with atom set the set $P(V):=\{\K v, v\in V\setminus\{0\}\}$ of one dimensional subspaces of $V$.
\end{Example}
We first give a useful preparation lemma on posets.
\begin{Lemma}\label{lemmaA6}
	Given a poset $(L, \leq)$ bounded from below by $0$, there is a one-one correspondence between
	\begin{itemize}\itemsep=0pt
		\item[$(1)$] {\it poset locality} relations $\mtop$ on $L$
$($cf.\ Definition~$\ref{de:locposet})$ such that $0\mtop x$ for all $x\in L $ with the {\it separating property}: for any $x\in L$, the set $x^\mtop$ admits a greatest element $\grt \big(x^\mtop\big)$.
		%\label{it:sym}
		\item[$(2)$] {\it antitone} maps $\Psi\colon L\to L$, $x\mapsto x'$ $($i.e., $x\leq y\Rightarrow y'\leq x'$$)$ such that $x\leq (x')'.$\footnote{Note that this is implied by the involutivity required of an orthocomplementation.}
		%\label{it:map}
	\end{itemize}
	The correspondence is given by $($cf.\ Theorem~$\ref{thm:complement_strongly_local})$
	\begin{itemize}\itemsep=0pt
		\item $\Psi^\mtop(x):=x':=\grt \big(x^\mtop\big)$,
		\item $x\mtop y$ if and only if $y\leq x'$.
	\end{itemize}
	
	Moreover, let $\Psi$ be an antitone as in $(2)$. For any subset $X\subset L$, if the supremum $\sup X$ of~$X$ in $L$ exists, then so does the infimum $\inf Y$ of the set
	$Y=\{x'\,|\, x\in X\}$, and~$(\sup X)'=\inf Y$.
\end{Lemma}
\begin{proof}
Given a poset locality $\top$ on $L$ in (1) and let $\Psi^\top (x):=\grt \big(x^\top\big)$. Then from $x\leq y$, we obtain $x^\top \supseteq y^\top$ and hence $\grt\big(x^\top\big) \geq \grt\big(y^\top\big)$.
Further, since $x\top \grt \big(x^\top\big)$, we have $x\in \big(\grt\big(x^\top\big)\big)^\top$ and thus
$x\leq \grt\big(\big(\grt\big(x^\top\big)\big)^\top\big).$

Given an antitone map $\Psi\colon L\to L$ in (2) and define $x\top_\Psi y$ if $y\leq x'=\Psi(x)$. The symmetry of~$\mtop$ follows from the antitone property: $y\leq x'$ implies $x\leq (x')'\leq y'$. Further $0\top_\Psi\, x$ for any $x\in L$ since $0$ is the smallest element. It also follows from the definition of $\top_\Psi$ that $\grt \big(x^{\top_\Psi}\big)=\Psi(x)$.

Now let $\Psi$ be an antitone as in (2) and $X\subseteq L$.
If $a=\sup X$ exists, then $x'\geq a'$ for all $x\in X$ and hence $a'$ is a lower bound of $Y$. Further, for any lower bound $z$ of $Y$, we have $z\leq x'$ for all $x\in X$. Then $x\leq (x')'\leq z'$. Thus $a=\sup X\leq z'$ and then $z\leq a'$. Therefore $a'=\inf Y$.
\end{proof}
\begin{Remark} \label{rk:posetltt}
If a poset $(L,\leq)$ is a lattice, then a poset locality relation on $L$ satisfying (1) is also a lattice locality relation, since the separating property means that $x^\top=\downarrow \grt\big(x^\top\big)$ and hence is a lattice ideal.
\end{Remark}
\begin{Proposition}
	Given an atomistic complete lattice $L$ with set $P$ of atoms.
	For $a\in L$, define $P_a:=\{p\in P\,|\, p\leq a\}=\downarrow a \cap P$.
	
	There is a one-one correspondence between lattice locality relations $\mtop_L$ on $L$ as in $(1)$ of Lemma~{\rm \ref{lemmaA6}} and
	\begin{itemize}\itemsep=0pt
		\item [$(3)$] symmetric binary relations $\top_P$ on $P$ such that for any $p\in P$ there is (unique) $a\in L$ such that $p^{\top_P}=P_a$.
		%One defines $\Psi^\top(p):=a$.
		%\label{it:perp}
	\end{itemize}
The correspondence is given by
	\begin{itemize}\itemsep=0pt
		\item $\top_P$ is the restriction of $\mtop_L$ to $P$,
		\item $\mtop_L$ is defined in terms of $\top_P$ by: $a\mtop_L b$ if and only if
		\begin{enumerate}\itemsep=0pt
\item $a=0$ or $b=0$, or

\item $a\neq 0$, $b\neq 0$, and $p\top_P\, q$ for all $p, q\in P$ with $p\leq a$ and $q\leq b$.
		\end{enumerate}
	\end{itemize}
\end{Proposition}
\begin{proof}
Given $\top_L$ as in (1) and let $\top_P$ be as defined. For any $p\in P$, $p^{\top_L}$ has the greatest element $a:=\grt \big(p^{\top_L}\big)$. Then $p^{\top_L}=\downarrow a$. Then $p^{\top_P}=p^{\top_L}\cap P= \downarrow a \cap P=P_a$. Thus $\top_P$ satisfies~(3).

Given $\top_P$ as in (3) and let $\top_L$ be as defined. Then $\top_L$ is obviously symmetric. Also $0\in L^{\top_L}$ by definition.

Now we prove that $a^{\top _L}$ is a lattice ideal for any $a\in L$. If $a=0$, then $0^{\top _L}=L$ which is a~lattice ideal. First consider $a=p\in P$. Then $p^{\top _P}=P_b$ for some $b\in L$. We~just need to show $p^{\top_L}=\downarrow b$.
	
First let $c\in \downarrow b$, that is, $c\le b$. Then for any $q\le c$ with $q\in P$, we have $q\in P_b$. So $q\top _P \ p$, thus $c\top _L \ p$, that is, $c\in p^{\top_L}$. This proves $\downarrow b \subseteq p^{\top_L}$.
On the other hand, let $c\in p^{\top_L}$, that is, $c\top_L\, p$. Then for any $q\leq c$ with $q\in P$, we have $q\top_P p$. Thus $q\in P_b$, that is, $q\in \downarrow b$, implying $c\in \downarrow b$ since $L$ is atomistic. This proves $p^{\top_L}\subseteq \downarrow b$.

As a consequence, we have $p\top_L b$ and $p\top_L c$ implying $p\top_L(b\vee c)$.

We now show that $a^{\top_L}$ is closed under $\vee$. For $a\top_L b, a\top_L c$ and $p\in P$ with $p\leq a$, we have $p\top_L b$ and $p\top_L\, c$. Thus $p\top_L\, (b\vee c)$ and hence $p\top_P q$ for any $q\in P$ with $q\leq b\vee c$. Therefore $a \top_L (b\vee c)$. Further, if $b\in a^{\top_L}$ and $c\leq b$, then for $p\in P_a$ and $q\in P_c$, we have $q\in P_b$ and hence $p\top_P\, q$. This gives $c\in a^{\top_L}$. Thus $a^{\top_L}$ is also a poset ideal and hence a lattice ideal.

We finally show
\begin{gather*}
a^{\top_L} = \downarrow \big(\vee\big\{p\in P\,|\, p\in a^{\top_L}\big\}\big).
\end{gather*}
The right hand side is defined since $L$ is complete. The inclusion $\supseteq $ holds since $a^{\top _L}$ is a lattice ideal. Further, for $b\in a^{\top_L}$, write $b=q_1\vee \cdots \vee q_k$ with $q_1,\ldots,q_k\in P$. Since $q_1,\ldots, q_k\in a^{\top_L}$, we have $b=q_1\vee \cdots \vee q_k\leq \vee\big\{p\in P\,|\, p\in a^{\top_L}\big\}$. This completes the proof of the equality.

In summary, the locality relation $\top_L$ from $\top_P$ satisfies (1).

The one-one correspondence results from the assignment
\begin{gather*}\top_P\mapsto \top_L:=\top_{P,L}\mapsto \top_{P,L,P},
\end{gather*}
which is clearly the identity since $\top_{P,L}$ extends $\top_P$ and the assignment
\begin{gather*}
\top_L\mapsto \top_P:=\top_{L,P}\mapsto \top_{L,P,L},
\end{gather*}
 using the conditions on $T_L$.
\end{proof}
\begin{Remark}
Applying the above proposition to the subspace lattice $L:=G(V)$ of a finite dimensional $\K$-vector space $V$ whose atom set $P(V)=\{\K v\,|\, v\in V\setminus \{0\}\}$ consists of one dimensional subspaces gives another proof of Proposition~\ref{pp:lvtoll}.
\end{Remark}

{\sloppy\subsection*{Acknowledgements} This work is supported by Natural Science Foundation of China (Grant Nos.\ 11771190, 11821001, 11890663). The first and third authors are grateful to the Perimeter Institute where part of~this paper was written.
They also thank Tobias Fritz for inspiring discussions {and the third author is grateful to Daniel Bennequin for his very useful comments at a preliminary stage of~the preparation of the paper}. We~greatly appreciate very helpful suggestions of the referees on~previous versions of the paper.

}

\pdfbookmark[1]{References}{ref}
\LastPageEnding

\end{document}